\documentclass[10pt]{amsart}

\usepackage{amsthm}
\usepackage{amssymb,amsmath,amsfonts,microtype}
\usepackage{graphicx}
\usepackage{pdfsync}

\newcommand{\comment}[1]{}

\usepackage{esint}

\newtheorem{thm}{Theorem}
\newtheorem{cor}{Corollary}
\newtheorem{lem}{Lemma}
\newtheorem{prop}{Proposition}

\newtheorem*{thm1'}{Theorem 1'}

\newtheorem*{thm2'}{Theorem 2'}

\newtheorem*{thmB'}{Theorem B$^\prime$}
\newtheorem*{thmA'}{Theorem A$^\prime$}
\newtheorem*{propnA'}{Proposition A$^\prime$}
\newtheorem*{propnB'}{Proposition B$^\prime$}
\theoremstyle{remark}

\theoremstyle{definition}

\theoremstyle{remark}

\DeclareMathOperator{\supp}{supp}
\DeclareMathOperator{\lcm}{lcm}
\newcommand{\R}{\mathbb{R}}

\newcommand{\Z}{\mathbb{Z}}
\newcommand{\N}{\mathbb{N}}
\newcommand{\FF}{\mathbb{F}}

\newcommand{\Q}{\mathbb Q}
\newcommand{\C}{\mathbb{C}}

\newcommand{\RR}{\mathcal{R}}
\newcommand{\MM}{\mathcal{M}}

\newcommand{\HH}{\mathcal{H}}

\newcommand{\PP}{\mathcal{P}}

\newcommand{\XX}{\mathcal X}
\newcommand{\EE}{\mathcal E}
\newcommand{\DD}{\mathcal D}

\newcommand{\MA}{\mathcal A}
\newcommand{\MB}{\mathcal B}
\newcommand{\Tr}{\Delta}

\newcommand{\la}{\lambda}
\newcommand{\lm}{\lambda}

\newcommand{\D}{\delta}
\newcommand{\La}{\Lambda}

\newcommand{\De}{\Delta}
\newcommand{\ga}{\gamma}
\newcommand{\Ga}{\Gamma}
\newcommand{\eps}{\varepsilon}

\newcommand{\te}{\theta}

\newcommand{\tnk}{\theta_{d,k}}
\newcommand{\Om}{\Omega}

\newcommand{\whf}{\widehat{f}}

\newcommand{\1}{\mathbf{1}}

\newcommand{\subs}{\subseteq}
\newcommand{\ls}{\lesssim}

\newcommand{\wh}{\widehat}

\newcommand{\be}{\begin{equation}}
\newcommand{\ee}{\end{equation}}
\newcommand{\eq}{\begin{equation}}
\newcommand{\bee}{\begin{equation*}}
\newcommand{\eee}{\end{equation*}}

\newcommand{\diag}{\operatorname{diag}}
\newcommand{\tr}{\operatorname{tr}}

\begin{document}

\title{
 Discrete maximal operators associated to simplices}
\author{Neil Lyall \quad\quad  \'{A}kos Magyar  \quad\quad Alex Newman \quad\quad Peter Woolfitt}
\thanks{The first and second authors were partially supported by grants NSF-DMS 1702411 and NSF-DMS 1600840, respectively.}

\address{Department of Mathematics, The University of Georgia, Athens, GA 30602, USA}
\email{lyall@math.uga.edu}
\email{magyar@math.uga.edu}
\email{alxjames@uga.edu}
\email{pwoolfitt@uga.edu}



\setlength{\parskip}{2pt}

\begin{abstract} 
We prove $\ell^2$ estimates for certain discrete maximal operators associated to simplices. These  operators are generalizations of the discrete spherical maximal operator.
\end{abstract}

\maketitle


\section{Introduction}


An important result in the development of discrete harmonic analysis is the $\ell^p$-boundedness of the so-called discrete spherical maximal function \cite{MSW}. 
For any $\la\in\sqrt{\N}$ we let $S_\lm=\{y\in\Z^d:\ |y|=\la\}$ denote the discrete sphere of radius $\lm$ centered at the origin. For $f:\Z^d\to\R$ we then define the discrete spherical averages 
\bee
\MA_\la f(x)= |S_\lm|^{-1} \sum_{y\in S_\lm} f(x+y).
\eee
 noting that if $d\geq 5$, then
$c_d\la^{d-2}\leq |S_\la| \leq C_d\la^{d-2}$ for some constants $0<c_d<C_d<\infty$, see \cite{V}.
In \cite{MSW} it was shown that for $p>d/(d-2)$ one has the following maximal function  estimate
\bee
\bigl\|\,\sup_{\lm\geq1}|\MA_\lm f|\bigr\|_{\ell^p(\Z^{d})} \leq C_{p,d}\,\|f\|_{\ell^p(\Z^{d})}
\eee
where $\|f\|_{\ell^p(\Z^{d})} = (\sum_x |f(x)|^p)^{1/p}$.

Given a non-degenerate $k$-simplex $\Tr=\{v_0=0,v_1,\ldots,v_k\}\subs\Z^d$ and $\lm\in\sqrt{\N}$, we let 
\[S_{\lm\Delta}:=\{(y_1,\ldots,y_k)\in \Z^{dk}:\ \Tr'=\{0,y_1,\ldots,y_k\}\simeq \la\Tr\}\]
noting that if $d\geq2k+3$ and  $\lm\in\sqrt{\mathbb{N}}$, then
\be\label{K}c_\De\,\la^{dk-k(k+1)}\leq \, |S_{\lm\Delta}|\,\leq C_\De\,\la^{dk-k(k+1)}\ee
for some constants $0<c_\De<C_\De<\infty$, see \cite{K} or \cite{Magy09}.  Recall that for any $1\leq k\leq d$ we refer to a configuration $\Delta=\{0,v_1,\ldots,v_k\}\subseteq\Z^d$ as a non-degenerate $k$-simplex if the vectors $v_1,\dots, v_k$ are linearly independent. 

For $f:\Z^{dk}\to\R$ we then define the discrete linear simplicial averaging operator
\bee
\MB_{\la\Delta} f(x)= |S_{\lm\Delta}|^{-1} \sum_{y\in S_{\lm\Delta}} f(x+y).
\eee
These operators are a direct analogue of the spherical operators, but with the sphere $S_\lm$ replaced by the surface $S_{\lm\Delta}$. 
For functions $f_1,\dots,f_k:\Z^d\to\mathbb{C}$ one could also define  the \emph{multilinear} simplicial averaging operator
\[
\MA_{\la\Delta}(f_1,\ldots,f_k)(x) = |S_{\lm\Delta}|^{-1} \sum_{(y_1,\dots,y_k)\in S_{\lm\Delta}} f_1(x+y_1)\cdots f_k(x+y_k).
\]

Note that for $k=1$ and $v_1=(1,0,\ldots,0)$ we have that $S_{\la\Tr}=S_\la$ and hence $\MB_{\la\Delta} f=\MA_{\la\Delta}f=\MA_{\la}f$. 

\subsection{Main Result}


Our main result is the following.

\begin{thm}\label{NewSimplex}
If $k\geq1$, $d\geq 2k+3$, and $\Delta=\{0,v_1,\dots,v_k\}\subseteq\Z^{d}$ be a non-degenerate $k$-simplex, then
\be\label{newfull}
\bigl\|\,\sup_{\lm\geq1}|\MB_{\lm\Delta}f|\bigr\|_{\ell^2(\Z^{dk})}\leq C_{d,\Delta} \|f\|_{\ell^2(\Z^{dk})}.
\ee
\end{thm}

Theorem \ref{NewSimplex} has the following immediate consequence.

\begin{cor}\label{OldSimplex}
If $k\geq1$, $d\geq 2k+3$, and $\Delta=\{0,v_1,\dots,v_k\}\subseteq\Z^{d}$ be a non-degenerate $k$-simplex, then
\be\label{oldfull}
\bigl\|\,\sup_{\lm\geq1}|\MA_{\lm\Delta}(f_1,\dots,f_k)|\bigr\|_{\ell^2(\Z^{d})}\leq C_{d,\Delta} \|f_1\|_{\ell^2(\Z^{d})}\cdots  \|f_k\|_{\ell^2(\Z^{d})}.
\ee
\end{cor}

\begin{proof}[Proof of Corollary \ref{OldSimplex}]
Given functions $f_1,\dots,f_k:\Z^d\to\mathbb{C}$ we define $f=\otimes_{j=1}^k f_j$ to be a tensor product. It follows that for any element $\tilde{x}$ on the diagonal $\mathbb{D} = \{(x,...,x)\in\Z^{dk}\,:\, x\in\Z^d\}$ we would have that 
\[\MB_{\la\Delta} f(\tilde{x})=\MA_{\la\Delta}(f_1,\ldots,f_k)(x)\]
and hence 
\bee
\bigl\|\,\sup_{\lm}|\MA_{\lm\Delta}(f_1,\dots,f_k)|\bigr\|_{\ell^2(\Z^{d})}
=\bigl\|\,\sup_{\lm}|(\MB_{\lm\Delta}f)1_{\mathbb{D}}|\bigr\|_{\ell^2(\Z^{dk})}
\leq\bigl\|\,\sup_{\lm}|\MB_{\lm\Delta}f|\bigr\|_{\ell^2(\Z^{dk})}.\qedhere
\eee
\end{proof}

The $\ell^p$ mapping properties of the maximal operators corresponding to the multilinear averages $\MA_{\lm\Delta}$ were considered in \cite{AKP} and \cite{CLM}. The $\ell^2\times\cdots\times\ell^2\to\ell^2$ estimate in Corollary \ref{OldSimplex} constitutes the first non-trivial estimates of any type for this maximal operator in dimensions lower than $d=2k+5$ when $k\geq2$.  

To our knowledge, Theorem \ref{NewSimplex} provides the first non-trivial estimates for the maximal operator corresponding to the averages $\MB_{\lm\Delta}$ when $k\geq2$.

\subsection{Key refined estimates}


Recall that for  $f\in \ell^1(\mathbb{Z}^{dk})$ we define its Fourier transform $\widehat{f}: \mathbb{T}^{dk}\to\mathbb{C}$ by
\[\widehat{f}(\xi)=\sum_{x\in\mathbb{Z}^{dk}}f(x)e^{-2\pi i x\cdot\xi}.\]

Following the approach in \cite{LMNW} we will deduce Theorem \ref{NewSimplex} from  refined estimates for our maximal operators at a single dyadic scale, namely Proposition \ref{MainProp} below.  We first need to introduce some notation. 
For any  integer $j\geq0$ we let
\[
q_j=\lcm\{1, 2, \dots,2^j\}
\]
noting that $q_j$ is of the order $e^{2^j}$, and for any non-negative integers $j$ and $l$ that satisfy $2^j\leq l$ , we let
\be\label{22}
\Omega_{j,l}:=\{\xi\in\mathbb{T}^{dk}: \xi \in[-2^{j-l},2^{j-l}]^{dk}+(q_j^{-1}\mathbb{Z})^{dk}\}.
\ee

 \begin{prop}\label{MainProp}
If $k\geq1$, $d\geq 2k+3$, and $\Delta=\{0,v_1,\dots,v_k\}\subseteq\Z^{d}$ be a non-degenerate $k$-simplex, then
 \begin{equation}\label{4*}
     \bigl\|\sup_{2^l\leq\lambda\leq2^{l+1}} |\MB_{\lm\Delta}f|\Bigr\|_{\ell^2(\Z^{dk})}\leq C_{d,\Delta} \,2^{-j/2}j^{-1}\, \|f\|_{\ell^2(\Z^{dk})}
 \end{equation}
whenever $\supp \widehat{f}\subseteq \Omega_{j,l}^c$, where $\Omega_{j,l}^c$ denotes the complement of $\Omega_{j,l}$.
 \end{prop}

It is easy to see that Proposition \ref{MainProp} is equivalent to the following refined ``mollified" estimates in which one obtains gains in $\ell^2$ over suitably large scales when applied to functions whose Fourier transform  is localized away from rational points with small denominators.

\begin{thm}\label{NewMollifiedSimplex}
If $k\geq1$, $d\geq 2k+3$, and $\Delta=\{0,v_1,\dots,v_k\}\subseteq\Z^{d}$ be a non-degenerate $k$-simplex, then for any 
$\eta>0$, and $L\geq q_\eta^4$, we have
\be\label{4}
\Bigl\|\sup_{\lm\geq\eta^{-2} L}|\MB_{\lm\Delta}f|\Bigr\|_{\ell^2(\Z^{dk})}\leq C_{d,\Delta} \frac{\eta}{\log \eta^{-1}}\, \|f\|_{\ell^2(\Z^{dk})}
\ee
whenever $\supp \widehat{f}\subseteq \Omega_{\eta,L}^c$, where \[\Omega_{\eta,L}=\{\xi\in\mathbb{T}^{dk}: \xi\in[-L^{-1},L^{-1}]^{dk}+(q_\eta^{-1}\mathbb{Z})^{dk}\}\] and $q_\eta:=\lcm\{1\leq q\leq \eta^{-2}\}$.
\end{thm} 

Indeed, note that in proving \eqref{4} one may restrict the supremum to $\eta^{-2} L\leq\lm\leq 2\eta^{-2} L$. Choosing $l,j\in\mathbb{N}$ such that $2^l\leq\eta^{-2} L\leq 2^{l+1}$ and $2^j\geq \eta^{-2}$ we have that  $2^{l-j}\leq L$ and hence $\Omega_{j,l} \subseteq \Omega_{\eta,L}$. 
Applying estimate (\ref{4*}) in Proposition \ref{MainProp} with $j$ and $l$ chosen as above implies estimate (\ref{4}) of Theorem \ref{NewMollifiedSimplex}, while applying estimate (\ref{4}) of Theorem \ref{NewMollifiedSimplex} with $L=2^{l-j}$ and $\eta=2^{-j/2}$ immediately implies estimate (\ref{4*}) of Proposition \ref{MainProp}.

Estimate (\ref{4}) in the case $k=1$ was originally established by the first two authors in \cite{LM19} via an adaptation of the transference methods from \cite{MSW}. 

\subsection{Outline of paper}

The deduction of Theorem \ref{NewSimplex} from Proposition \ref{MainProp} follows exactly as in Section 3 of \cite{LMNW}, just with $\R^d$ and $\Z^d$ replaced with $\R^{dk}$ and $\Z^{dk}$, as such we choose to omit these details. In Section \ref{ProofProp} below we reduce the proof of Proposition \ref{MainProp} to two key estimates for theta functions on the Siegel upper half space, these estimates are established in Section \ref{Appendix}. 

\section{Proof of Proposition \ref{MainProp}}\label{ProofProp}

Given any simplex $\Tr=\{v_0=0,v_1,\ldots,v_k\}\subseteq\R^d$, we introduce the associated \emph{inner product matrix} $T=T_\Tr= (t_{ij})_{1\leq i,j\leq k}$ with entries $t_{ij}:= v_i\cdot v_j$, where ``$\cdot$"
stands for the dot product in $\R^d$. Note that $T$ is a positive semi-definite matrix with integer entries and $T$ is positive definite if and only if $\Tr$ is non-degenerate. 

It is easy to see that $\Tr'\simeq \la\Tr$, with $\Tr'=\{y_0=0,y_1,\ldots,y_k\}$, if and only if 
\eq\label{2.4}
y_i\cdot y_j=\la^2 t_{ij}\quad \text{for all}\quad 1\leq i,j\leq k.
\ee
If we let $M\in\Z^{d\times k}$ be a matrix with column vectors $y_1,\ldots,y_k\in\Z^d$, then the system of equations above can be written as the matrix equation
\eq\label{3.1}
M^t M =\la^2 T,
\ee
where $M^t$ is the transpose of the matrix $M$. It therefore follows that
\[
\MB_{\la\Delta}f(x) = |S_{\lm\Delta}|^{-1} \sum_{y\in \Z^{dk}} f(x+y)S_{\la^2 T}(M)
\]
if we use $S_{\la^2 T}(M)$ to denote the indicator function of relation \eqref{3.1}.

Let $I_k=[0,2]^{k(k+1)/2}$ denote the space of symmetric $k\times k$ matrices with entries in the interval $[0,2]$. Using the fact that 
\[\tr (X^t Y) =\tr (Y X^t) =\sum_{i=1}^k \sum_{j=1}^k x_{ij}y_{ij},\]
for any $k\times k$ matrices $X=(x_{ij})$, $Y=(y_{ij})$, one has
\eq\label{3.2}
S_{\la^2 T}(M)= 2^{-k} \int_{I_k} e^{\pi i\,\tr[(M^t M -\la^2 T) X]} \,dX
\ee
where $dX=\prod_{1\leq i\leq j\leq k} dx_{ij}$. Moreover, if $M^t M=\la^2 T$ then 
\[\tr (T^{-1} M^t M) = \tr (M T^{-1} M^t) =\tr (\la^2 I) = k\la^2.\]

Given $l\in\N$ write $\La=2^l$ and $\eps =2^{-2l}$. We have 
\eq\label{3.3}
S_{\la^2T}(M)= 2^{-k} e^{k\eps\la^2}\,
\int_{I_k} e^{-\pi i \la^2 \tr(TX)}\,e^{\pi i\,\tr (M(X+i\eps T^{-1})M^t)} dX.
\ee

Let \[G_{X,\eps}(M) = G_{X,\eps}(y_1,\ldots,y_k)=e^{\pi i\,\tr(M(X+i\eps T^{-1})M^t)}\]
be the Gaussian function, where $y_1,\ldots,y_k \in\Z^d$ are the column vectors of the matrix $M$, and define the corresponding operator
\eq\label{3.4}
B_{X,\eps} f(x) := \sum_{y_1,\ldots,y_k\in\Z^d}
f(x+(y_1,\dots,y_k))\,G_{X,\eps}(y_1,\ldots,y_k).
\ee

It follows that 
\[\MB_{\la\Delta}f(x) = 2^{-k} e^{k\eps\la^2} |S_{\la\Tr}|^{-1}
\int_{I_k} e^{-\pi i \la^2 \tr(TX)}\, B_{X,\eps} f(x)\,dX.\]
Thus for the maximal function 
\[\MM_lf:=\sup_{2^l\leq\lambda\leq2^{l+1}}|\MB_{\lm\Delta}f|\]
we have the pointwise estimate
\eq\label{3.5}
\MM_l f(x)\leq C_{d,\Delta}\, \La^{-k(d-k-1)}\,\int_{I_k} |B_{X,\eps} f(x)|\,dX
\ee
as $\eps=\La^{-2}=2^{-2l}$ and $\La\leq \la\leq 2\La$. Finally, by Minkowski's inequality
\eq\label{3.6}
\|\MM_l f\|_{\ell^2(\Z^{dk})}\leq C_{d,\Delta}\, \La^{-k(d-k-1)} \,\int_{I_k} \|B_{X,\eps} f(x)\|_{\ell^2(\Z^{dk})}\,dX.
\ee
Since 
\[\widehat{B_{X,\eps} f}=\widehat{G_{X,\eps}}\widehat{f}\]
it follows from Plancherel's identity that
\eq\label{3.8}
\|B_{X,\eps} f\|_{\ell^2(\Z^{dk})}\leq\,
 \|\widehat{G_{X,\eps}}\|_\infty\ \|f\|_{\ell^2(\Z^{dk})}
\ee
and hence that the $\ell^2\times\cdots\times \ell^2\to \ell^2$ boundedness of the dyadic maximal operator $\MM_l f$ will follow from the estimate
\eq\label{3.9}
\int_{I_k} \|\widehat{G_{X,\eps}}\|_\infty\,dX\leq C_{d,\Delta}\, \La^{k(d-k-1)}\ee
with $\La=2^l$. For the refined estimate we use the assumption that $\supp\,\whf\subs \Om_{j,l}^c$, that is
$\whf =1_{\Om_{j,l}^c} \whf$, thus in order to prove Theorem 2 it is enough to show that, for $j,l\in\N$ with $2^{j+2}\leq l$, one has 
\eq\label{3.10}
\int_{I_k} \|1_{\Om_{j,l}^c}(\xi_1)\, \widehat{G_{X,\eps}}(\xi)\|_\infty\,dX \leq C_{d,\Delta}\, 2^{-j/2}j^{-1}\,\La^{k(d-k-1)} \ee
with $\La=2^l$.


\section{Estimates for theta functions on the Siegel upper half space}\label{Appendix}


To prove estimates \eqref{3.9} and \eqref{3.10} we will follow the approach given in Section 5 of \cite{Magy09}. For the sake of completeness we recall below some of the basic notions and constructs.
If $M=[m_1,\ldots,m_k]\in\Z^{d\times k}$ and $\XX=[\xi_1,\dots,\xi_k]\in\R^{d\times k}$ are $d\times k$ matrices, then one has that $\tr (M^t \mathcal{X})=m_1\cdot\xi_1+\ldots +m_k\cdot\xi_k$ where $\cdot$ denotes the usual dot product. Thus, the Fourier transform of a function $f(m_1,\ldots,m_k)=f(M)$ may written 
\[\wh{f}(\XX)=\wh{f}(\xi_1,\ldots,\xi_k) = \sum_{M\in \Z^{d\times k}} 
f(M) e^{-2\pi i \tr( M^t\XX)}.\]
This implies that 
\eq\label{4.1} 
\wh{G}_{X,\eps}(\XX) = \sum_{M\in \Z^{d\times k}} e^{\pi i\,\tr[(M(X+i\eps T^{-1})M^t -2M^t \XX]} = \theta_{d,k} (X+i\eps T^{-1},-\XX,0)
\ee
where $\te_{d,k}:\mathbb{H}_k\times\R^{d\times k}\times\R^{d\times k}\to\mathbb{C}$ is the theta-function defined by
\eq\label{4.2}\te_{d,k}(Z,\XX,\mathcal{E}) = \sum_{M\in \Z^{d\times k}} e^{\pi i\,\tr[(M-\EE)Z(M-\EE)^t +2M^t \XX-\EE^t \XX]}
\ee
for $Z=X+iY\in\mathbb{H}_k$, with $\mathbb{H}_k$ being the Siegel upper space, see (5.1)-(5.3) in \cite{Magy09}. 



We partition the range of integration $I_k$ and
estimating the theta function separately on each part by
exploiting its transformation properties. This may be viewed as the extension of the classical Farey arcs decomposition to $k>1$. Recall the integral symplectic group
\eq\label{4.3x}\Ga_k=\left\{\ga=\left(%
\begin{array}{cc}A & B \\C & D \\\end{array}%
\right):\ AB^t=BA^t,\ CD^t=DC^t,\ AD^t-BC^t=E_k, \right\}\ee
which acts on the Siegel upper-half space
$\,\HH_k=\{Z=X+iY:\ X\in \mathcal{M}_k, Y\in\PP_k\}\,$ as a group of analytic automorphisms. The action being
defined by $\ga \langle Z\rangle =(AZ+B)(CZ+D)^{-1}\ $ for
$\ga\in\Ga_k,\,Z\in\HH_k$, see \cite{Magy09} and also \cite{K}. 
Let us recall also the subgroup of
integral modular substitutions
\eq\label{4.4x}\Ga_{k,\infty}=\left\{\ga=\left(%
\begin{array}{cc}A & B \\0 & D \\\end{array}%
\right):\ AB^t=BA^t,\ AD^t =E_k \right\}\ee

Writing $U=A^t$ and $S=AB^t$, it is easy to see that $D=U^{-1}$
and $B=SU^{-1}$, moreover $S$ is symmetric and $U\in GL(k,\Z)$,
i.e. $\,\det(U)=\pm 1$. The action of such
$\ga\in\Ga_{k,\infty}$ on $Z\in\HH_k$ takes the form
\eq\label{4.5x} 
\ga
\langle Z\rangle  = Z[U]+S
\ee
using the notation $Z[U]=U^t ZU$. The general linear
group $GL(k,\Z)$ acts on the space $\PP_k$ of positive $k\times k$
matrices, via the action $Y\to Y[U]$ for $Y\in\PP_k$, and let $\RR_k$
denote the corresponding so-called Minkowski domain, see Definition 1 on page 12 of  \cite{KL}. A matrix $Y=(y_{ij})\in\RR_k$ is called
reduced. We recall that for a reduced matrix $Y$ with $y_{11}\leq y_{22}\leq\cdots\leq
y_{kk}$
\eq\label{4.6x} Y\approx Y_D\ee
where $Y_D=\diag(y_{11},\ldots,y_{kk})$ denotes the diagonal part
of $Y$, and $A\approx B$ means that $A-c_k B>0$ and $B-c_k A>0$ for
some constant $c_k>0$. For a proof of these facts, see Lemma
2 on page 20 in \cite{KL}. A fundamental domain $\DD_k$ for the action of $\Ga_k$
on $\HH_k$, called the Siegel domain,  consists of all matrices
$Z=X+iY$, ($X=(x_{ij})$), satisfying
\eq\label{4.7x} Y\in\RR_k,\ \ \ |x_{ij}|\leq 1/2,\ \ \ \ |\det\,(CZ+D)|\geq 1,\ \ \ \forall\  \ga=\left(%
\begin{array}{cc}A & B \\C & D \\\end{array}%
\right)\in\Ga_k.\ee

The second rows of the matrices $\ga\in\Ga_k$ are parameterized by
the so-called coprime symmetric pairs of integral matrices
$(C,D)$, which means that $CD^t$ is symmetric and the matrices
$GC$ and $GD$ with a matrix $G$ of order $k$ are both integral
only if $G$ is integral, see Lemma 2.1.17 in \cite{A}. It is clear from
definition (5.6) that if $\ga_2=\ga\ga_1$ with second rows
$(C_2,D_2)$ and $(C_1,D_1)$ for some $\ga\in\Ga_{k,\infty}$, then
$(C_2,D_2)=(UC_1,UD_1)$ for some $U\in GL(k,\Z)$. On the other
hand, if both $\ga_1$ and $\ga_2$ have the same second row $(C,D)$
then $\ga_2\ga_1^{-1}\in\Ga_{k,\infty}$. This gives the
parametrization of the group $\Ga_{k,\infty}\backslash\Ga_k$ by
equivalence classes of coprime symmetric pairs $(C,D)$ via the
equivalence relation $(C_2,D_2)\sim(C_1,D_1)$ if
$(C_2,D_2)=(UC_1,UD_1)$ for some $U\in GL(k,\Z)$, see also 
page 54 in \cite{A}. We will use the notation $[\ga]=[C,D]\in
\Ga_{k,\infty}\backslash\Ga_k$.

If one defines the domain:
$\FF_k=\cup_{\ga\in\Ga_{k,\infty}} \ga \DD_k$, then $\
\HH_k=\bigcup_{[\ga]\in \Ga_{k,\infty}\backslash\Ga_k}\ \ga^{-1}
\FF_k$ is a non-overlapping cover of the Siegel upper half-plane.
Correspondingly, for a given matrix $T>0$ of order $k$, define the
Farey arc dissection of level $T$, as the cover
\eq\label{4.8x} I_k=\bigcup_{[\ga]\in \Ga_{k,\infty}\backslash\Ga_k}\ I_T
[\ga],\ \ \ \ I_T [\ga]=\{X\in I_k:\ X+iT^{-1}\in\ga^{-1}
\FF_k\}\ee

We recall the basic estimates (5.14)-(5.16) in \cite{Magy09} whose proofs are based on the transformation property 
\[ |\tnk (Z,\XX,0)|= |\det\,(CZ+D)|^{-\frac{d}{2}}\ |\tnk
(\ga\langle Z\rangle,\, \XX A^t-K_\ga/2,\, \XX C^t-N_\ga/2)|\]
for some matrices $K_\ga, N_\ga\in\Z^{n\times k}$, see Proposition 5.2 in \cite{Magy09}.
Namely, if $(C,D)$ is a coprime symmetric pair, then
for $Z\in I_T[C,D]$ one has
\eq\label{4.9x} |\tnk (Z,\XX,0)|\leq C_{d,k}\, |\det\,(CZ+D)|^{-\frac{d}{2}}\ee
\noindent
uniformly for $\XX\in\MM_k(\R)$.

Next we describe the ``mollified" estimate (5.16) in \cite{Magy09} in slightly different form. For $q\in\N$ and $\tau>0$ define the region
\eq\label{4.10x}
\Om_{q,\tau} =\{\XX\in\R^{d\times k}:\ |\XX-P/2q|\leq \tau\ \ \text{for some}\ \ P\in\Z^{d\times k}\}.\ee

If $[\ga] = [C,D]$ is  a coprime symmetric pair, $q:=|\det(C)|>0$, then for $Z\in I_T[C,D]$
\eq\label{4.11x}|\tnk (Z,\XX,0)|\ls
|\det\,(CZ+D)|^{-\frac{d}{2}}\,\left(e^{-c\min(Y)}+e^{-c\,\tau^2\mu (C^t Y C)}\right)
\ee
\noindent
uniformly for $\XX\in\Om_{q,\tau}^c$. Here $Y=\text{Im}\, \ga\langle Z\rangle$, $\min(Y)=\min_{x\in\Z^d, x\neq 0} |Y x\cdot x|$ and $\mu(Y)=\min_{x\in\R^d,\,|x|=1} |Y x\cdot x|$.


Define, similarly as in (5.20) in \cite{Magy09}
\eq\label{4.3}
J_T[C,D] = \int_{I_T[C,D]} \sup_\XX |\theta_{d,k} (X+i T^{-1},-\XX,0)|\,dX.
\ee
By \eqref{4.9x} we have that 
\eq\label{4.4}
J_T[C,D] \leq C_{d,k}\ J_T^0[C,D],
\ee
where 
\eq J^0_T[C,D]=\int_{X\in I_T[C,D]} |\det (CZ+D)|^{-\frac{d}{2}}\,dX.\ee
If $q:=|\det(C)|>0$, then for $\tau>0$ let
\eq\label{4.6}
J_{T,\tau}[C,D]: =\int_{I_T[C,D]}\ \sup_\XX \1_{\Om_{\tau,q}^c}(\XX)\, |\theta_{d,k} (X+i T^{-1},-\XX,0)|\,dX.
\ee
By estimate \eqref{4.11x} one has 
\eq\label{4.7}
J_{T,\tau}[C,D] \leq C_{d,k}\ J_T^1[C,D] + J_{T,\tau}^2[C,D],
\ee
where 
\eq J^1_T[C,D]=\int_{I_T[C,D]} |\det
(CZ+D)|^{-\frac{d}{2}}\,e^{-c\,min(Y)}\,dX
\ee 
\eq
J^2_{T,\tau}[C,D]=\int_{I_T[C,D]} |\det
(CZ+D)|^{-\frac{d}{2}}\,e^{-c\tau^2\,\mu(C^t YC)}\,dX.
\ee
where $Y=Im\,\ga\langle Z\rangle$ and $\ga\in\Ga_k$ such that
$[\ga]=[C,D]\in \Ga_{k,\infty}\backslash\Ga$. 

Then by inequalities (5.24)-(5.26) given in Propositions 5.3-5.4 in \cite{Magy09}, we have 
\eq\label{4.8}
\sum_{S^t=S} J_T[C,D+CS] \leq C_{d,k}\ \det(T)^{\frac{d-k-1}{2}} |\det (C)|^{-\frac{d}{2}}
and
\ee
\eq\label{4.9}
\sum_{S^t=S} J_{T,\tau}[C,D+CS] \leq C_{d,k}\ \det(T)^{\frac{d-k-1}{2}}
\big(|\det (C)|^{-k} \min(T)^{-\frac{d-2k}{4}} + |\det(C)|^{-\frac{d}{2}} (\tau^2 \mu(T))^{-\frac{d-2k}{4}}\big)
\ee
where the summation is over all symmetric integral matrices $S\in \MM_k(\Z)$. 

Recall that the map $[C,D]\to C^{-1}D$ provides a one-one and onto
correspondence between the classes of coprime symmetric pairs
$[C,D]\in \Ga_{k,\infty}\backslash\Ga_k$, with $\det (C)\neq 0$, and 
symmetric rational matrices $R$ of order $k$, and the pairs $[C,D+CS]$ correspond to the matrices $R+S$ with symmetric $S\in\Z^{k\times k}$. Let us write $\Q(1)^{k\times k}$ for the space of modulo 1 incongruent symmetric rational matrices, where $\Q(1)=\Q/\Z$, $\Q$ being the
set of rational numbers. If $R=C^{-1}D$, for a coprime symmetric pair
$[C,D]$ then will write
\eq\label{4.10}
J_T[R] := \sum_{S^t=S} J_{T}[C,D+CS],
\ee
\eq\label{4.11}
J_{T,\tau}[R] := \sum_{S^t=S} J_{T,\tau}[C,D+CS],
\ee
which is well-defined as it only depends on the equivalence class $[R]\in \Q(1)^{k\times k}$.
Finally write $d(R)=|\det(C)|$ for $R=C^{-1}D$. Then by \eqref{4.2} and \eqref{4.3}, we have with $\eps=\La^{-2}$ that
\eq\label{4.12}
\int_{I_k} \sup_\XX |\theta_{d,k} (X+i \eps T^{-1},-\XX,0)|\,dX 
= \sum_{[C,D], \det(C)\neq 0} J_{\La^2 T}[C,D] + \sum_{[C,D], \det(C)= 0} J_{\La^2 T}[C,D] =: \sum_1 + \sum_2.
\ee

An estimate for the second sum is given in Corollary 5.1 in \cite{Magy09}, namely it is shown that 
\eq\label{4.13}
\sum_2\leq C_{d,k}\ |\La^2 T|^{(k-1)(d-k)/2}\leq C_{d,k}\La^{(d-k)(k-1)}
\ee
where $|T|=(\sum_{ij} t_{ij}^2)^{1/2}$ is the Euclidean norm of the matrix $T$. For the first sum we use estimate \eqref{4.8} for the matrix $\La^2 T$, which implies
\eq\label{4.14}
\sum_1 = \sum_{[R]\in \Q(1)^{k\times k}} J_{\La^2 T}[R]\leq C_{d,k}\ \La^{k(d-k-1)} \sum_{[R]\in \Q(1)^{k\times k}} d(R)^{-d/2}.
\ee

Recall the following estimate, proved in Lemma 1.4.9 in \cite{K};
for $u\geq 1$ and $s>1$ one has 
\eq\label{4.15} u^{-s} \sum_{1\leq d(R)\leq u}
d(R)^{-k}+ \sum_{d(R)\geq u} d(R)^{-k-s}\leq C
(2+\frac{1}{s-1})\,u^{1-s}
\ee 
where the summation is taken over $[R]\in \Q(1)^{k\times k}$. In particular 
$\sum_R d(R)^{-d/2}\ls 1$ in dimensions $d>2k+2$, thus estimate \eqref{3.9} follows from \eqref{4.1}, \eqref{4.12} and estimates \eqref{4.13}-\eqref{4.14}. 
 
For the mollified estimate \eqref{3.10}, we set $\tau=2^{j-l}$ besides $\La=2^l$ and $\eps=2^{-2l}$. Again, we note that if $q = |\det (C)|>0$ and if $q\mid q_j$ i.e. if $q$ divides $q_j$ then $\xi_1\in \Om_{j,l}^c$ implies that $\XX\in \Om_{\tau,q}$ for $\XX=(\xi_1,\ldots,\xi_d)$, for the sets $\Om_{j,l}$ and $\Om_{\tau,q}$ defined in \eqref{22} and \eqref{4.10x}. Using this observation, we have  
\eq\label{4.16}
\int_{I_k} \sup_\XX \1_{\Om_{j,k}^c}(\xi_1)|\theta_{d,k} (X+i \eps T^{-1},-\XX,0)|\,dX\ 
\ls\ \sum_{d(R)\mid q_j} J_{\La^2 T,\tau}[R] + \sum_{d(R)\nmid q_j} J_{\La^2 T}[R] +\sum_2.
\ee
In dimensions $d\geq 2k+3$, using \eqref{4.9} and \eqref{4.15}, the first sum on the right side of \eqref{4.16} is crudely estimated by  
\begin{align} \label{4.17}
\sum_{d(R)\mid q_j} J_{\La^2 T,\tau}[R] &\ \ls\ \La^{k(d-k-1)} \sum_{1\leq d(R)\leq q_j} \big( d(R)^{-k} \La^{-\frac{d-2k}{2}} + d(R)^{-\frac{d}{2}} (\tau\La)^{-\frac{d-2k}{2}}\big)\\
&\ls\  \La^{k(d-k-1)} \big(q_j 2^{-\frac{3l}{2}} + 2^{-\frac{3j}{2}}\big)\ \ls\ \La^{k(d-k-1)} 2^{-\frac{3j}{2}}\nonumber.
\end{align}

Indeed, $q_j=\lcm\{1\leq q\leq 2^j\}\approx e^{2^j}\leq 2^l$ as $2^{j+2}\leq l$ by our assumptions. To estimate the second term on the right side of \eqref{4.16}, we need the following.

\begin{lem}\label{L1} Let $j\in \N$ and $s>1$. Then 
\eq\label{4.18}
\sum_{d(R)\nmid q_j} d(R)^{-k-s}\leq C\, 2^{j(1-s)}\,j^{-1}
\ee
where the constant $C$ may depend on $d,k$ and $s$.
\end{lem}

\begin{proof} Let 
\eq\label{4.19}
\Psi(s):= \sum_{[R]\in \Q(1)^{k\times k}} d(R)^{-k-s} = \sum_{n\geq 1} a_k(n) n^{-s},
\ee
\noindent
with $a_k(n)= \sum_{d(R)=n} d(R)^{-k}$. For two Dirichlet series $\Psi(s)=\sum_{n\geq 1} a(n) n^{-s}$ and $\Phi(s)=\sum_{n\geq 1} b(n) n^{-s}$ we will write $\Psi(s)\preceq \Phi(s)$ if $|a(n)|\leq b(n)$ for all $n\geq 1$. 

It is proved in \cite{K} that 
\eq\label{4.20}
\Psi(s) \preceq \zeta(s+1)^K \zeta(s) =: \sum_{n\geq 1} b_K(n) n^{-s},
\ee
with $K=2^k+k-3$, see $(34)$ in Lemma 1.4.9. Clearly the coefficients of the Dirichlet series $\zeta(s+1)^K \zeta(s)$ are multiplicative i.e. $b_K(n m)=b_K(n) b_K(m)$ if $(n,m)=1$, moreover are easy to show that, 
\eq\label{4.21}
b_K(n)= \sum_{m\mid n} \frac{d_K(m)}{m},
\ee
where $d_K(m)=|\{m_1,\ldots,m_k\in\N: m_1 m_2\cdots m_K=m\}|$. 
Since $q_j=l.c.m.\{1\leq q\leq 2^j\}$, if $n\nmid q_j$ the either there is a prime $p>2^j$ such that $p\mid n$ or there is a prime $p< 2^j$ such that $p^{\ga_p}>2^j$
but $p^{\ga_p}\mid n$. Accordingly, we have the estimate
\eq\label{4.22}
\sum_{d(R)\nmid q_j} d(R)^{-k-s} = \sum_{n\nmid q_j} a_k(n) n^{-s} 
\leq  \sum_{p>2^j} \sum_{n\geq 1} b_K(pn) p^{-s} n^{-s} + \sum_{p<2^j} \sum_{n\geq 1} b_K(p^{\ga_p} n) p^{-\ga_p s} n^{-s}.
\ee
Writing $n=p^r m$, the first sum on the right side of \eqref{4.22} is estimated by
\eq\label{4.23}
\sum_{p>2^j} \sum_{n\geq 1} b_K(pn) p^{-s} n^{-s}= \sum_{p>2^j} \sum_{r=1}^\infty \sum_{m\geq 1, p\nmid m} b_K(p^r) b_K(m) p^{-rs} m^{-s},
\ee 
using the fact that $b_K(p^r m)=b_K(p^r) b_k(m)$. By \eqref{4.21}, we have 
\eq\label{4.24}
b_K(p^r) = 1+\sum_{s=1}^r \frac{d_K(p^s)}{p^s}\leq 1+\sum_{s=1}^\infty \frac{(s+1)^K}{2^s} \ls 1,
\ee 
uniformly in $r\geq 1$.
Thus, for $s>1$,
\eq\label{4.25}
\sum_{p>2^j}\sum_{r=1}^\infty \sum_{m\geq 1, p\nmid m} b_K(p^r) b_K(m) p^{-rs} m^{-s}\ \ls\ \sum_{p>2^j} p^{-s} \ \ls\ 2^{j(1-s)} j^{-1},
\ee
using the fact that the number of primes $2^J\leq p<2^{J+1}$ is bounded by $2^J\,J^{-1}$ for all $J\geq j$.

The second term on the right side of \eqref{4.22} is estimated similarly, except that here we use the fact that $p^{\ga_p}>2^j$ for $p<2^j$. We have 
\begin{align}\label{4.26}
    \sum_{p<2^j} \sum_{n\geq 1} b_K(p^{\ga_p} n) p^{-\ga_p s} n^{-s} 
&= \sum_{p<2^j}\sum_{r=\ga_p}^\infty \sum_{m\geq 1, p\nmid m} b_K(p^r) b_K(m) p^{-rs} m^{-s}\\
& \ls\ \sum_{p<2^j}\sum_{r=\ga_p}^\infty p^{-rs}\ \ls\ \sum_{p<2^j} p^{-\ga_p s}\ \ls\ 2^{j(1-s)} j^{-1},\nonumber
\end{align}
as the number of primes $p<2^j$ is bounded by $2^j j^{-1}$. Estimate \eqref{4.18} follows immediately from \eqref{4.25}-\eqref{4.26}.
\end{proof}

In dimensions $d>2k+2$, Lemma \ref{L1} with $s=d/2-k\geq 3/2$ implies that
\eq\label{4.27}
\sum_{d(R)\nmid q_j} J_{\La^2 T}[R]\ \ls\ \La^{k(n-k-1)} d(R)^{-d/2}\ \ls\ \La^{k(n-k-1)} 2^{-j/2} j^{-1},
\ee
with $\La=2^l$. Finally, by \eqref{4.13} \eqref{4.16}-\eqref{4.17} and \eqref{4.27} one obtains, in dimensions $d>2k+2$
\eq\label{4.28}
\int_{I_k} \sup_\XX \1_{\Om_{j,k}^c}(\xi_1)|\theta_{d,k} (X+i \eps T^{-1},-\XX,0)|\,dX\ 
\ls \La^{k(d-k-1)} \big(2^{-j/2} j^{-1} + 2^{-3j/2} + 2^{-3l}\big)\ \ls\ \La^{k(d-k-1)} 2^{-j/2} j^{-1}.
\ee
\medskip

Estimate \eqref{3.10} follows immediately from \eqref{4.1} and \eqref{4.28}.

\comment{

\bigskip

\section{We could add this...}

We could add another 3-4 pages to establish the following weak combinatorial consequence:

\begin{thm}\label{thm1}
\label{Pinned}
Let $k\geq1$, $A\subseteq\Z^d$ with $d\geq 2k+3$, and $\Delta=\{0,v_1,\dots,v_k\}\subseteq\Z^{d}$ be a non-degenerate $k$-simplex.
If $\D^*(A)>0$, 
there exists an integer $q=q(\D^*(A))$ and  $\lm_0=\lm_0(A, \Delta)$ such that  for any $\lm_1\geq \lm_0$ there exists fixed points $x_1,\dots,x_k\in A$ such that for all $\lm\in[\lm_0,\lm_1]\cap\sqrt{\N}$  one has $x_1+y_1,\dots,x_k+y_k\subseteq  A$ for some $\{0,y_1,\dots,y_k\}\simeq\lm q\Delta$.
\end{thm}

\begin{cor}
\label{Pinned}
Let $k\geq1$, $A\subseteq\Z^d$ with $d\geq 2k+3$, and $\Delta=\{0,v_1,\dots,v_k\}\subseteq\Z^{d}$ be a non-degenerate $k$-simplex.
If $\D^*(A)>0$, 
there exists an integer $q=q(\D^*(A))$ and  $\lm_0=\lm_0(A, \Delta)$ such that  for any $\lm_1\geq \lm_0$ there exists a fixed $x\in \underbrace{A\times\cdots\times A}_{\text{$k$ copies}}$ such that for all $\lm\in[\lm_0,\lm_1]\cap\sqrt{\N}$  one has $x+\Delta'\subseteq  \underbrace{A\times\cdots\times A}_{\text{$k$ copies}}$
for some $\Delta'=\{0,y_1,\dots,y_k\}\simeq\lm q\Delta$.
\end{cor}

** I don't think it is worth it. What don you think?
}


\end{document}